\documentclass[preprint,prX,linenumbers,amsmath,amssymb]{amsart}
\usepackage{float}
\usepackage{graphicx}
\usepackage{epsfig}
\usepackage{epstopdf}
\usepackage{dcolumn}
\usepackage{bm}
\usepackage{bm,color}
\usepackage{float}
\usepackage{ mathrsfs }
\usepackage{amsmath}
\usepackage{amssymb}
\usepackage{amsmath}
\newtheorem{lemma}{Lemma}
\newtheorem{proposition}{Proposition}
\newtheorem{theorem}{Theorem}
\newtheorem{corollary}{Corollary}

\newtheorem{remark}{Remark}

\def\R{\mathbb R}

\def\N{\mathbb N}
\usepackage[bordercolor=white, color=white]{todonotes}

\def\p{\partial}

\usepackage{mathtools}

\DeclarePairedDelimiter\norm{\lVert}{\rVert}
\begin{document}
\title[Conditional stability for SIV]{Stability estimate for scalar image velocimetry}
\author{E. Burman  and J.~J.~J.~Gillissen and L. Oksanen}
\address{
Department of Mathematics,
University College London,\\
Gower Street, London, WC1E 6BT.\\
E-mail: e.burman@ucl.ac.uk,\, jurriaangillissen@gmail.com, l.oksanen@ucl.ac.uk}
\date{\today}
\begin{abstract}
In this paper we analyse the stability of the system of partial
differential equations modelling scalar image velocimetry.
We first revisit a successful numerical technique to reconstruct velocity vectors ${u}$ from images of a passive scalar field $\psi$
by minimising a cost functional, 
that penalises the difference between the reconstructed scalar field $\phi$
and the measured scalar field $\psi$, under the constraint that
$\phi$ is advected by the reconstructed velocity field ${u}$, 
which again is governed by the 
Navier-Stokes equations. We investigate the stability of the
reconstruction by
applying this method to synthetic scalar fields in two-dimensional turbulence, 
that are generated by numerical simulation. Then we present a
mathematical analysis of the nonlinear coupled problem and prove that,
in the two dimensional case, smooth solutions of the Navier-Stokes
equations are uniquely determined by the measured scalar field.
We
also prove a conditional stability estimate showing that 
the map from the measured scalar field $\psi$ to 
the reconstructed velocity field $u$, on any
interior subset, is H\"older continuous.
\end{abstract}
\maketitle
\section{Introduction}
\label{introduction}
Scalar image velocimetry (SIV) is the reconstruction of the fluid velocity field  ${u}$ 
from measurements of a scalar field $\psi$, that is advected by
${u}$. 
This idea, which dates back to the works \cite{FV84, Wunsch85,Wunsch87}, is applied in weather forecasting models 
using e.g. satellite images of clouds or ocean temperature. For a
background of the technique see for instance
\cite{kalnay2003atmospheric} and references therein. 

SIV also finds applications in medical flow imaging and in experimental fluid mechanics.
For instance in the laser induced fluorescence (LIF) technique, 
a fluid is seeded with fluorescent molecules 
and laser light is focused into a thin sheet, where it is absorbed by the fluorescent molecules 
followed by spontaneous emission of light which is then recorded by a
camera (see e.g. \cite{su1996scalar} and references therein). 
Assuming that the recorded light intensity is proportional to the 
fluorescence concentration $\psi$, 
the velocity field  ${u}:=u_1 e_x + u_2 e_y$, with $e_x,\, e_y$ the
Cartesian unit vectors, can be reconstructed by
invoking the scalar transport equation,
    \begin{align}\label{eq_for_psi}
\partial_t \psi
+u \cdot \nabla\psi -\lambda\Delta \psi =0,
    \end{align}
where $\lambda$ is the scalar diffusivity. A direct inversion of the scalar transport equation 
only provides the component $u^\bot$ of ${u}$ that is normal to $\psi$ isolines,
    \begin{align*}
{u}^{\bot}=(-\partial_t\psi+\lambda\Delta \psi) {\nabla} \psi/\vert {\nabla} \psi\vert^2.
    \end{align*}
Finding all components of ${u}$ requires additional constraints, for example, 
conservation of hydrodynamic variables. 

Inspired by recent developments on computational methods
for SIV \cite{gillissen2018space},
in this work we study
mathematically the stability of the map from $\psi$ to $u$. We show, in particular,
that under suitable a priori assumptions, the velocity field is
uniquely determined by the measured scalar field, thereby giving a
partial theoretical justification to the computational approaches. The
analysis uses conditional stability estimates as a workhorse for understanding of inverse problems in a spirit similar to that of the
seminal work by Bukhge\u{\i}m and Klibanov \cite{BK81}.

To fix a configuration for the stability analysis we consider a soap film experiment from \cite{gillissen2018space}.
In that work, a soap film with a thickness of $\sim 10\mu$ m is formed in the 10 cm gap between two vertical parallel nylon wires.
Turbulence in the film is generated by piercing the film with an array of cylindrical obstacles,
resulting in chaotically interacting wake vortices. These perturbations
are accompanied by film thickness fluctuations which behave similar as those of a passive scalar. 
The soap film is illuminated and light reflections are recorded at high speed.
The recorded interference pattern depends on the film thickness fluctuations 
and therefore behaves similar as a passive scalar. 

Motivated by this experiment, we consider a
particular Cartesian geometry, $\Omega = (-a,a) \times (-b,b)$, with $a,b > 0$, and
write $Q=\Omega \times (0,T)$ for $T > 0$. 
Our proofs generalize to many other two dimensional settings but the above choice allows for convenient notations.
We assume that the flow satisfies the slip or no-slip
conditions on the vertical boundaries, $x = \pm a$, 
and that the scalar field satisfies the no flux condition there, 
but the boundary conditions for both the velocity $u$ and the scalar field $\psi$ on the top
and bottom boundaries $y = \pm b$ are unknown. 

Determination of $u$ given $\psi$ is possible only if $\psi$ satisfies some non-degeneracy condition.
Indeed, a constant $\psi$ satisfies (\ref{eq_for_psi}) but gives no information on $u$.
To state our main result in a simplified form (see Theorem \ref{cor_NS_unique_cont} below for the precise formulation),
we assume that the spatial derivative of $\psi$ does not vanish identically on the left boundary $x=-a$ at any time.

The simplified formulation is as follows.
Let ${u}$  and $\tilde { u}$ be smooth velocity
fields satisfying the Navier-Stokes equations, together with, say, the no-slip boundary condition on $x=\pm a$,
and let $\psi$ and $\tilde \psi$ be smooth scalar fields satisfying (\ref{eq_for_psi})
with $u$ and $\tilde u$, respectively, together with the no flux boundary condition on $x = \pm a$.
Suppose, moreover, that for all $t \in (0,T)$ there is $y \in (-b,b)$ such that 
$\p_y \psi(-a,y,t) \ne 0$. 
Then for every space time domain $\omega$, such
that $\overline \omega \subset Q$ (here and below $\overline X$
denotes the closure of the set $X$), there exists $\alpha \in (0,1)$ and $C > 0$ such that the following stability estimate holds
\begin{equation}\label{eq:1st_bound}
\|{u} - \tilde {u}\|_{L^2(\omega)}\leq C \|\psi - \tilde \psi\|_{H^4(Q)}^\alpha,
\end{equation}
whenever $u$ and $\tilde u$ satisfy an a priori bound.
Observe that this estimate in particular shows that the scalar field $\psi$
uniquely determines the velocity field $u$. It also gives an upper bound
similar to those applied in \cite{BBFV20, BO18} to obtain sharp error
estimates for finite element data assimilation methods. Although the
estimate \eqref{eq:1st_bound} can not be directly applied in that
context, it is a first step towards an analysis of the error
propagation in computational velocity reconstruction based on SIV.

Our approach combines a stream function formulation for the two dimensional
Navier-Stokes' equations, with the classical pressure velocity
formulation. 
First, using a stream function $\Theta$ to represent the
velocity field $u$, the convection--diffusion equation
for the scalar field $\phi$ defines a transport equation for $\Theta$. Provided the
scalar field satisfies the non-degeneracy condition $\p_y \psi \ne 0$ on the left boundary, 
this transport
equation can be solved in a neighbourhood of
the left boundary. This way we show that $u$ can be reconstructed locally given the scalar field $\psi$. 
Then to extend this local
reconstruction to an arbitrary subdomain in the interior of the space time domain, we apply a unique continuation result for the transient
linearized Navier-Stokes' equations \cite{LW20}.
Observe that the local
reconstruction step is possible whenever the velocity $u$ is known
along a curve segment transverse to the level curves of $\psi$. In our
configuration this holds on the lateral boundaries thanks to the
chosen boundary conditions. A variant of this local reconstruction
approach, using classical techniques, has recently been applied to SIV in \cite{SRMH19}. 


The outline of the paper is as follows. First in section
\ref{sec:SIV_comp} we recall a recent technique for computational
velocity reconstruction,
and study its stability computationally in Section~\ref{sec:turbulence}. Then in Section
\ref{sec:analysis} we prove the conditional stability estimate. The
paper ends with concluding remarks in Section \ref{sec:conclusion}.

\section{An example of a computational method for SIV}\label{sec:SIV_comp}

In this section we describe a variational method
for the reconstruction of the space and time dependent velocity field 
${u}$, pressure field $p$ and scalar field $\phi$
from a measured space and time dependent scalar field $\psi$,
also referred to as the true scalar field.
We assume perfect knowledge of $\psi$ over a 
space time volume of $\Omega\times(0,T)$, where $\Omega$ and $(0,T)$ represent the corresponding spatial and temporal dimensions.
The reconstruction method divides $(0,T)$ into segments.
The the $i$-th segment starts and ends at $t_0=(i-1)\tau$ and $t_1=t_0+\tau$, respectively,
where $\tau$ is referred to as the segment time. 
The reconstruction scheme solves a sequence of optimisation problems 
for the unknown state variable ${w}=({u},p,\phi)$ 
at the start of each segment, i.e. at $t=t_0$.
We use a subscript on a field variable to indicate a time instance, 
e.g. ${w}_0={w}(t_0)$. 
Finding ${w}_0$ in each segment involves an iterative scheme, and
the initial guess for the iteration is taken from the reconstructed field 
${w}_1$ at $t=t_1$ obtained in the preceding segment.

\subsection{The cost function and its minimisation}
\label{scheme}
Defining the state variable as:
${w}=({u},p,\phi)$,
the velocity reconstruction method finds the initial conditions for the state variable ${w}_0$ in each time segment, 
by minimising the deviation between the reconstructed scalar field $\phi$ 
and the measured scalar field $\psi$ integrated over $\Omega\times(0,T)$.
The corresponding cost functional for the method reads:
\begin{equation}
\mathscr{J}=\frac{1}{2}\int_{t_0}^{t_1}\| \phi-\psi \|^2dt,
\label{eq1}
\end{equation}
where the norm 
$\|\cdot\|$ is based on the $L^2(\Omega)$ inner product, denoted by $\langle \cdot,\cdot\rangle$.

Equation (\ref{eq1}) is minimised under the constraint that ${w}_1$ is related to ${w}_0$
via the conservation equations of fluid momentum, fluid mass and scalar field:
\begin{equation}
{R}({w})=
\left(
\begin{array}{ccc}
\partial_t {u}+{u}\cdot{\nabla}{u}+{\nabla}p-\nu\Delta {u} \\
{\nabla\cdot u} \\
\partial_t \phi+{u}\cdot{\nabla}\phi-\lambda\Delta \phi \\
\end{array}
\right)={0} .
\label{eq7}
\end{equation}
Here $\nu$ is the fluid kinematic viscosity and $\lambda$ is scalar diffusivity. 
Adding constraint (\ref{eq7}) to equation (\ref{eq1}) using the Lagrange multiplier 
$\hat{{w}}=(\hat{{u}},\hat{p},\hat{\phi})$ 
results in the following constrained cost functional, which is referred to as the Lagrangian $\mathscr{L}$:
\begin{equation}
\mathscr{L}=
\int_{t_0}^{t_1} 
\left(\frac12 \| \phi-\psi \|^2+
\langle\hat{{w}},{R}({w})\rangle
\right)dt.
\label{eq2}
\end{equation}

Minimising $\mathscr{L}$ of (\ref{eq2}) w.r.t. ${w}_0$ involves computing the gradient of $\mathscr{L}$ w.r.r ${w}_0$, i.e. $\delta \mathscr{L}/\delta{w}_0$,
following \cite{gillissen2018space} we obtain:
\begin{equation}
\frac{\delta \mathscr{L}}{\delta {w}_0}=
\left(
\begin{array}{c}
\frac{\delta \mathscr{L}}{\delta {u}_0}\\
\frac{\delta \mathscr{L}}{\delta p_0}\\
\frac{\delta \mathscr{L}}{\delta \phi_0}
\end{array}
\right)
=
\left(
\begin{array}{c}
-\hat{{u}}_0\\
0\\
- \hat{\phi}_0
\end{array}
\right).
\label{eq11}
\end{equation}
This expression for 
$\delta \mathscr{L}/\delta {w}_0$ contains the Lagrange multiplier 
$\hat{{w}}=(\hat{{u}},\hat{p},\hat{\phi})$,
whose evolution equation and initial and final conditions read:
\begin{subequations}
\begin{equation}
\left(
\begin{array}{ccc}
-\partial_t
  \hat{{u}}-{u}\cdot\left({\nabla}\hat{{u}}+{\nabla}\hat{{u}}^T\right)-{\nabla}\hat{p}-\nu
  \Delta
  \hat{{u}}-\phi{\nabla}\hat{\phi} \\
-{\nabla\cdot \hat{u}} \\
-\partial_t \hat{\phi}-{u}\cdot{\nabla}\hat{\phi}-\lambda \Delta \hat{\phi}+\phi-\psi \\
\end{array}
\right)
={0},
\end{equation}
\begin{equation}
\hat{{u}}_1={0},
\hspace{1 cm}
\hat{\phi}_1=0,
\end{equation}
\begin{equation}
\hat{{u}}_0={0},
\hspace{1 cm}
\hat{\phi}_0=0.
\end{equation}
\label{eq32}
\end{subequations}
Equation (\ref{eq32}a) governs the evolution of the Lagrange multiplier  
$\hat{{w}}=(\hat{{u}},\hat{p},\hat{\phi})$, showing that  
$\hat{{u}}$ is 
incompressible 
and is forced along gradients of $\hat{\phi}$, 
and that 
both $\hat{{u}}$ and $\hat{\phi}$ are 
advected by ${u}$ and are subjected to diffusion.
The diffusion coefficients $-\nu$ and $-\lambda$ of these transport equations 
are negative, and therefore these equations are integrated backward 
in time from $t=t_1$ to $t=t_0$. 
The `starting' 
conditions $\hat{{w}}_1$ at $t=t_1$ are given by equation (\ref{eq32}b), while the 
`final' conditions $\hat{{w}}_0$ at $t=t_0$ define 
the optimisation update direction of ${w}_0$ via (\ref{eq11}).
This direction approaches zero, 
when ${w}_0$ reaches an extremum of $\mathscr{L}$, 
which corresponds to the condition of (\ref{eq32}c).


To find ${w}_0$ we use the Polak-Rebiere variant of the conjugate gradient method 
\cite{polak1971computational},
which updates 
$ {w}_0$ along a search direction ${h}$ related to ${\delta \mathscr{L}}/{\delta {w}_0}$. 
The initial guess for $ {w}_0$ is $ {w}_1$ from the previous time segment, and the 
step length along ${h}$ is varied using 
Brent's line minimisation algorithm \cite{brent2013algorithms},
until the minimum of the functional $\mathscr{J}$, from (\ref{eq1}),
in this direction is found. 
The conjugate gradient algorithm is continued until the relative change in $\mathscr{J}$ 
between two consecutive iterations drops below 0.01. 
A reconstruction typically require 
$\sim 10^2$ conjugate gradient steps and $\sim 10$
Brent minimisation steps per conjugate gradient step.
Therefore the computational effort of both methods 
is equivalent to that of $\sim 10^3$ computational fluid dynamics simulations.

It was shown in \cite{gillissen2018space} that the reconstruction method method produces unstable results when 
the segment time $\tau$ is larger than the flow correlation time $\mathscr{T}$, 
which for the system described above corresponds to 
$\tau\gtrsim 0.1$.
The instability is related to the 
ill posed-ness of the initial value problem for chaotic systems, 
which corresponds to the cost functional developing 
multiple minima, when $\tau$ exceeds $\mathscr{T}$, see for instance \cite{pires1996extending}.
In order to stabilise the method for these cases regularisation terms
may be added 
to the functional of (\ref{eq1}).
%
%
In this work we restrict ourselves to $\tau\lesssim 0.1$, which does not require the use of these regularisation terms. 

\section{Numerical investigation of the stability of SIV}
\label{sec:turbulence}
\subsection{Setup}
\label{setup}
We apply the computational velocity reconstruction 
to a two dimensional, incompressible, decaying turbulent flow field 
in a space time domain $\Omega\times(0,T)$.
The spatial domain $\Omega=(-\pi,\pi)^2$ is a bi-periodic square with
side length $L=2\pi$ and the temporal domain 
$(0,T)$ has a size of $T=8$, which is referred to as the reconstruction time.
The objective is to reconstruct the velocity field ${u}$ 
from the measured space and time dependent scalar field $\psi$.
To distinguish between the reconstructed velocity field and the reference velocity field, giving the advection of $\psi$,
we denote the former by $u$ and the latter by $v$.
It is re-emphasised that we have access to perfect information of $\psi$ on $\Omega\times(0,T)$.

Both $\psi$ and ${v}$ start from random initial conditions and 
the initial conditions are normalised, such that $\mathscr{U}=\|{v}\|=1$ and $\|\psi\|=1$ at $t=0$. 
The diffusivity is $\lambda=2\times 10^{-3}$ and the viscosity is $\nu=1\times 10^{-3}$, 
which corresponds to a Reynolds number of $\text{Re}=\mathscr{U}L/\nu=6.3\times 10^3$
based on the initial velocity scale $\mathscr{U}$ 
and a Schmidt number of $\text{Sc}=\nu/\lambda=1/2$.
We furthermore use a segment time of $\tau=8\times 10^{-2}$.

The true velocity field ${v}$ and the true (measured) scalar field $\psi$ 
are generated on a $256^2$ grid by numerically integrating (\ref{eq7}).
Subsequently the scalar measurement is interpolated to a $128^2$ grid and 
the reconstruction fields ${u}$ and $\phi$ are obtained on this $128^2$ grid 
by iteratively integrating equations (\ref{eq7}, \ref{eq11}, \ref{eq32}). 
On both the $256^2$ and $128^2$ grids, 
equations (\ref{eq7}, \ref{eq32}a) are advanced in time 
using a computational time step of $\Delta t=10^{-3}$.
Spatial derivatives in these equations are computed using the Fourier
basis functions. Time integration is performed using the second-order
explicit Adams--Bashforth scheme for the advection terms and the
second-order implicit Crank--Nicolson scheme for the diffusion terms.
For further details on the numerical methods see
\cite{gillissen2019two,gillissen2019data,gillissen2018space}.

\subsection{Results}
\begin{figure}
\centerline{
\psfig{file=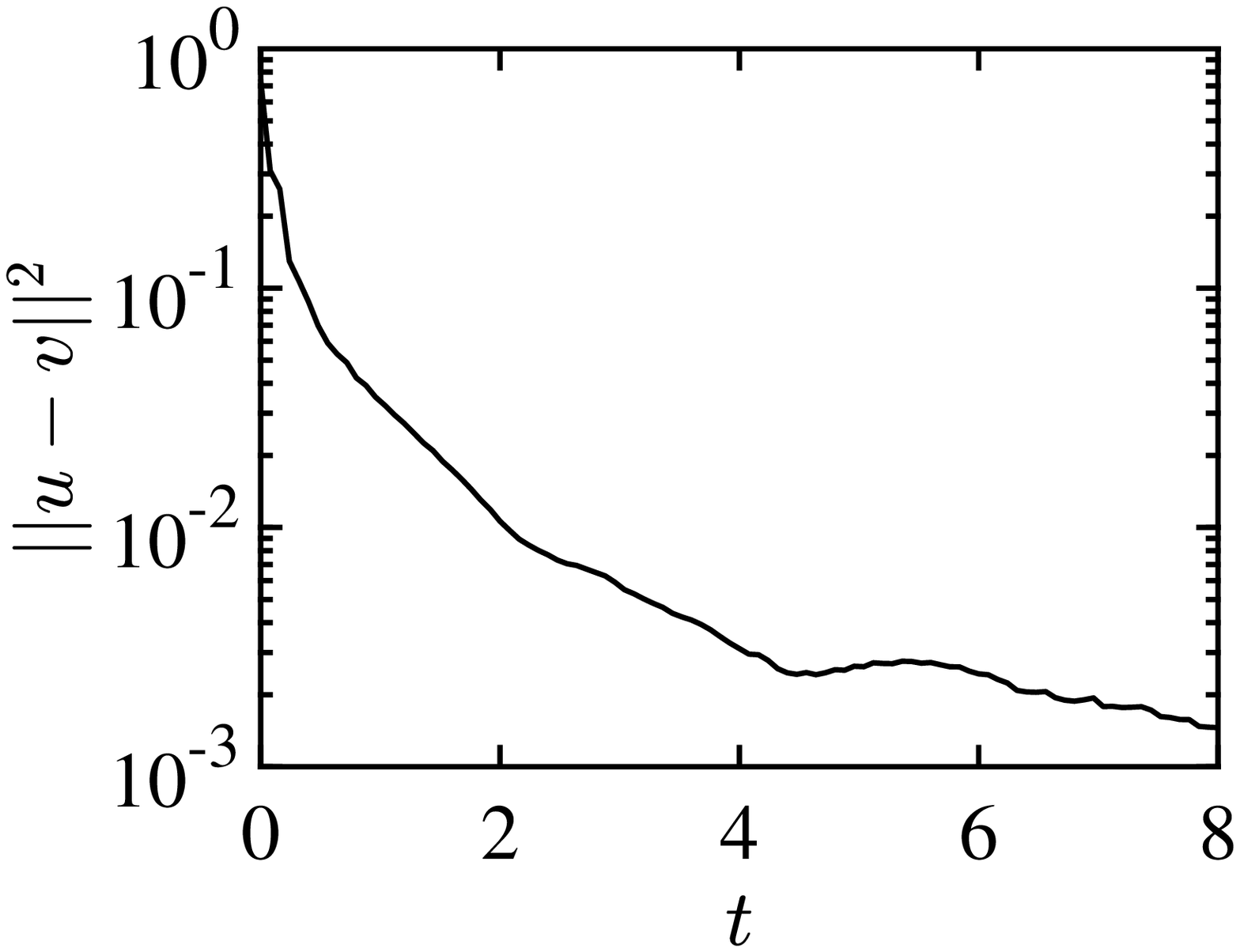,width=0.5\linewidth}
\psfig{file=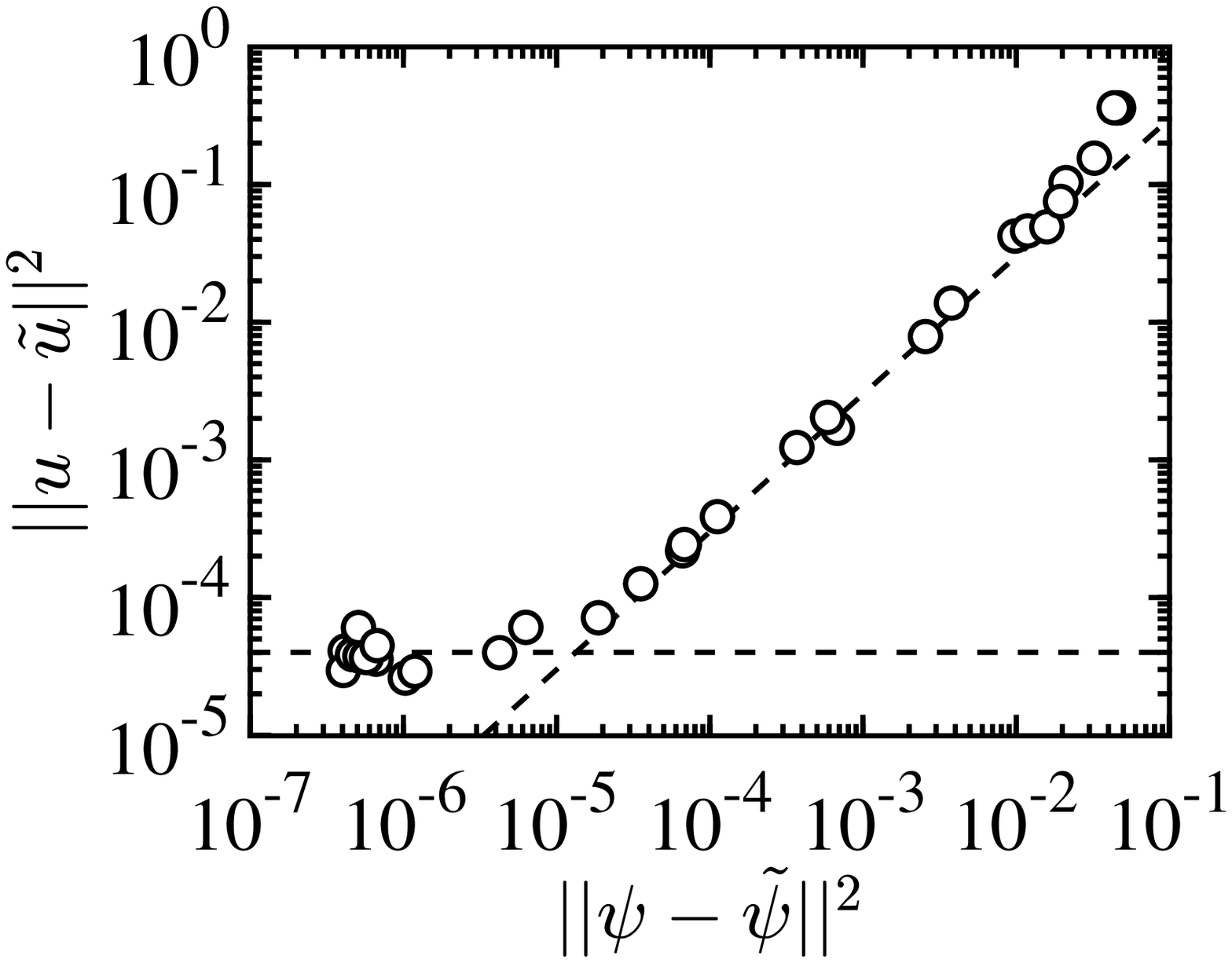,width=0.5\linewidth}
}
\caption{
(Left) Reconstruction error 
versus time. 
(Right) Variation in reconstructed velocity versus variation in measured scalar.
}
\label{fig1}
\end{figure}

\begin{figure}
\centerline{
\psfig{file=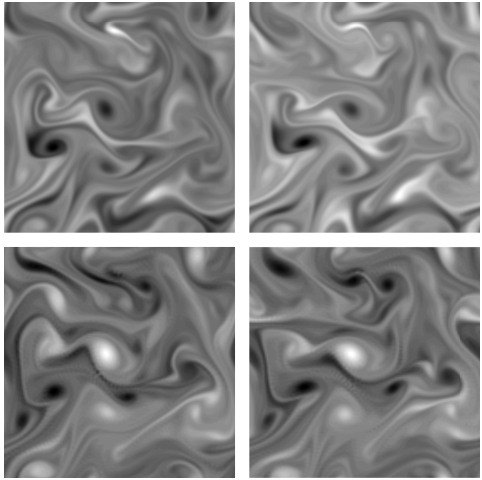,width=0.5\linewidth}
}
\caption{
(Top) Two `measured' scalar fields, $\psi$ and $\tilde{\psi}$, at $t=6$, 
that are advected by two `true' velocity fields ${v}$ and $\tilde{{v}}$. 
At $t=0$ it was assumed that $\tilde{\psi}=\psi$ and that $\tilde{{v}}={v}+\delta {v}$
where $\| \delta{v}\|/\|{v}\|=10^{-1}$.
(Bottom) The curl of the corresponding reconstructed velocity fields, 
${u}$ and $\tilde{{u}}$, at $t=6$.
}
\label{fig2}
\end{figure}
Figure \ref{fig1}(left) shows 
the reconstruction error 
$\epsilon={\| {u} - {v}\vert \vert^2}$ 
 as a function of time,
where it is recalled, 
that ${u}$ is the reconstructed velocity and ${v}$ is the reference velocity field. 
It is seen, that with time, $\epsilon$ first decreases 
and after $t> 4$, the error levels off.

Recalling the time segmentation described in the first paragraph if Section \ref{sec:SIV_comp}, 
the time dependent behaviour in Figure \ref{fig1}(left) 
indicates that the reconstruction depends 
on the quality of the initial guess for the initial condition at the start of each segment. 
In the first segment the initial guess is zero, while in each consecutive segment 
the initial guess becomes closer to the reference solution, explaining the 
observed decrease in $\epsilon$ with time. 
Also recall that 
the stability of the unique continuation problem for parabolic
equations degenerates towards
$t=0$, so a similar effect is expected also for globally coupled solutions. 

We now study the convergence of the reconstructed ${u}$ as a function of  the true (measured) $\psi$.
To this end we generate multiple simulations of true
scalar fields $\psi$ that are advected by different `true' velocity fields ${v}$.
For these simulations the conditions at $t=0$ for $\psi$ are identical but the conditions for ${v}$ differ.
Differences in a quantity $q$ between two simulations are denoted by 
$q-\tilde{q}$.
The relative difference in the initial conditions 
is varied between $10^{-5}$ and $10^{-1}$.
To study the convergence of ${u}$ as a function of $\psi$, we plot 
in Figure \ref{fig1}(right)  $\| {u}-\tilde{{u}}\|^2$ as a function 
of  $\| \psi-\tilde{\psi}\|^2$. 
It is noted that we average this norm over a time interval of $6<t<8$. 
During this time interval the computational method has converged roughly to a steady state, 
as show in Figure \ref{fig1}(left). 

Figure \ref{fig1}(right) shows that there are two regimes in the dependence of the variation of the 
reconstructed velocity $\| {u}-\tilde{{u}}\|^2$ on the variation of 
the true scalar $\| \psi-\tilde{\psi}\|^2$:
\begin{equation}
\| {u}-\tilde{{u}}\|^2 \sim
\left\{
\begin{array}{ccc}
\mathrm{constant} & \mathrm{for} & \| \psi-\tilde{\psi}\|^2\lesssim 10^{-5} \\
\| \psi-\tilde{\psi}\|^2 & \mathrm{for} & \| \psi-\tilde{\psi}\|^2\gtrsim 10^{-5} \\
\end{array}
\right. .
\label{eq8}
\end{equation}
These regimes are indicated by the dashed lines in Figure \ref{fig1}(right) which have a slope of zero and one, respectively. 
In these regimes, the variation 
of the true scalar $\| \psi-\tilde{\psi}\|^2$
(and of the true velocity $\| {v}-\tilde{{v}}\|^2$) 
is smaller and larger, respectively, than the error of the reconstructed velocity $\| {u}-{{v}}\|^2$.
Consequently, the variation of the reconstructed velocity $\| {u}-\tilde{{u}}\|^2$ in these two regimes is set 
by $\| {u}-{{v}}\|^2$ (constant) and by $\| \psi-\tilde{\psi}\|^2$,
respectively. The level of stagnation depends on the time and space
resolution of the Navier-Stokes' solver, which limits the accuracy of the reconstructed velocity.
The second regime of Eq. (\ref{eq8}) corresponds
to Lipschitz stability of the velocities. 
In view of the classical stability result \cite{Emanuilov1995} for data assimilation problems subject to the heat equation, Lipschitz stability
is feasible thanks to the fact that only the initial data of $u$ are
unknown in this computational example. We also show that the local reconstruction step is Lipschitz stable, see Remark \ref{rem_local} below. 
We leave a precise numerical analysis of the computational error as a topic of a future work,
see \cite{BO18} for such an analysis of data assimilation problems subject to the heat equation.

For illustration purposes,
we show in Figure \ref{fig2}(top) 
two true scalar fields, $\psi$ and $\tilde{\psi}$, at $t=6$, 
that are advected by two `true' velocity fields ${v}$ and $\tilde{{v}}$. 
At $t=0$ it was assumed that $\tilde{\psi}=\psi$ and that $\tilde{{v}}={v}+\Delta {v}$
where $\| \Delta{v}\|/\|{v}\|=10^{-1}$.
Figure \ref{fig2}(bottom) shows the curl of the corresponding reconstructed velocity fields, 
${u}$ and $\tilde{{u}}$, at $t=6$.

\section{Stability analysis}\label{sec:analysis}

In this section we will derive a stability estimate that gives a
mathematical interpretation of the stability illustrated by the right plot of
Figure \ref{fig1}. Compared to the numerical example we consider a more
challenging setting where not only the initial data but also the
boundary data on parts of the boundary are unknown. 

To fix the ideas we consider a simplified geometric configuration similar to that of a
two-dimensional soap film experiment, described above. Let $\Omega = (-a,a) \times
(-b,b)$ 
and write $Q = \Omega \times (0,T)$. The vertical boundaries of
the domain are defined by
   \begin{align*}
\Sigma = \Sigma_- \cup \Sigma_+, 
\quad \Sigma_\pm = \{ (x,y) : \text{$x = \pm a$ and $y \in (-b,b)$}\}.
    \end{align*}

Let $u$ be a solution to the Navier--Stokes equations in $Q$:
\begin{equation}\label{eq_NS}
\left\{
\begin{array}{r}
\p_t u - \nu \Delta u + (u \cdot \nabla) u + \nabla p = 0,
\\
\nabla \cdot u = 0.
    \end{array}
\right.
\end{equation}
On $\Sigma$ we assume that $u$ satisfies either homogeneous Dirichlet
conditions, $u=0$, or slip conditions, i.e. $u \cdot n = 0$, where $n$
is the outward pointing normal on $\Sigma$. On the top and bottom boundaries the
boundary data are unknown. The viscosity coeffcient $\nu > 0$ is known.

We suppose that a passive scalar $\psi$ satisfies
    \begin{align}\label{eq_phi}
\p_t \psi + u \cdot \nabla \psi - \lambda \Delta \psi =0
    \end{align}
in $Q$, together with the no flux boundary condition 
    \begin{align}\label{eq_phi_bd}
u_1 \psi - \lambda \p_x \psi = 0 \mbox{ on } \Sigma.
    \end{align}
Observe that since $u_1=0$ on $\Sigma$ this reduces to a Neumann
condition in practice. The diffusivity coefficient $\lambda > 0$ is a known constant.

Recall that the inverse problem to find $u$ given $\psi$ is clearly unsolvable if $\psi$ is a constant function, so we
make the standing assumption that this is not the case. More precisely
we will assume that $\psi$ is non-constant in space on $\Sigma$.

It is well-known that the Navier-Stokes equations admit smooth solutions in the two dimensional case, see e.g. \cite[Remark 3.7, p. 303]{Temam1977}.
For simplicity, we assume that both $u$ and $\psi$ are smooth in
$\overline \Omega \times (0,T)$.

Note that smoothness of $\psi$ in the interior of $Q$ actually follows from the interior Schauder estimates for (\ref{eq_phi}) and smoothness of $u$, see e.g. \cite[Theorem 8.12.1, p. 131]{Krylov1996}.
However, as we do not assume any boundary conditions on the horizontal boundaries $y = \pm b$,
we do not enter into discussion of smoothness properties of $\psi$.

\subsection{Stream function}
As $\Omega$ is a two-dimensional simply connected domain, the vanishing of divergence $\nabla \cdot
u = 0$ implies that there is $\Theta$ such that 
\begin{equation}\label{eq:vel_theta}
u = (\p_y \Theta, -\p_x \Theta).
\end{equation}
Observe that (\ref{eq_phi}) implies 
$$
\p_y \Theta \p_x \psi - \p_x \Theta \p_y \psi = \zeta
$$
where $\zeta = \zeta_\psi = -\p_t \psi + \lambda \Delta \psi$.
Defining the (time-dependent) vector field 
    \begin{align}\label{def_X}
X = X_\psi = -\p_y \psi\, \p_x + \p_x \psi\, \p_y
    \end{align} 
on $\Omega$, we can write equivalently
    \begin{align}\label{eq_transp}
X \Theta = \zeta,
    \end{align}
viewing the vector field X as the differential operator
defined by (\ref{def_X}).

As $\p_y \Theta = u_1 = 0$ on $\Sigma$, due to the boundary condition on $u$, and as the stream functions $\Theta$ and $\Theta + c$, with $c \in \R$ a constant,
give the same $u$, we may assume that 
    \begin{align}\label{eq_Theta_bd}
\Theta|_{\Sigma_-} = 0.
    \end{align}
We can view (\ref{eq_transp})--(\ref{eq_Theta_bd}) as a transport equation for $\Theta$. 
Observe that the vector field $X$ and the right-hand side $\zeta$ are known as $\psi$ is known. 

Suppose that 
$\p_y \psi(p_0) \ne 0$ for a point
$p_0 \in \Sigma_- \times (0,T)$.
Then near $p_0$, we can rewrite (\ref{eq_transp}) as
    \begin{align}\label{eq_transp2}
\p_x \Theta + \beta \p_y \Theta = f
    \end{align}
where $\beta = -(\p_y \psi)^{-1} (\p_x \psi)$ and $f = (\p_y \psi)^{-1}\zeta$.
Together with the boundary condition (\ref{eq_Theta_bd}), this transport equation can be solved near $p_0$.
In particular, recalling \eqref{eq:vel_theta}, we see that $\psi$ determines $u$ near $p_0$.
We will next study the continuity of the map $\psi \mapsto u$ near $p_0$. Observe that this map is non-linear, both $\beta$ and $f$ depend on $\psi$.

\subsection{On stability of linear transport equations}

In this section we study linear transport equations in an abstract
setting and the meaning of the variables here are different from those
above. 

Consider the equation
    \begin{align}\label{transport}
\p_t u + \beta \cdot \nabla u = f
    \end{align}
with the initial condition $u|_{t=0} = 0$.
We are interested in the continuity properties of the map $(f,\beta) \mapsto u$.
However, let us first recall the standard continuity result for the
map $f \mapsto u$ with $\beta$ fixed. The exposition
below was inspired by \cite[Theorem 7, p. 131]{Evans10}
where more complicated Hamilton--Jacobi equations are considered. Contrary to this reference, we need to keep track of the dependence on $\beta$ of the constants in the estimates. 

Let $T > 0$, suppose that $\beta \in L^\infty(0,T; W^{1,\infty}(\R^n; \R^n))$
and choose 
    \begin{align*}
R \ge \norm{\beta}_{L^\infty((0,T) \times \R^n; \R^n)}.
    \end{align*}
Furthermore, define $r(t) = R(T-t)$, let $B(r) \subset \R^n$ be the
open ball of radius $r$, with a fixed centre and outward pointing
normal $n_B$, and consider the energy
    \begin{align*}
E(t) = \int_{B(r(t))} u^2(t,x)\, dx.
    \end{align*}
Then for a solution $u$ of (\ref{transport}),
    \begin{align}\label{energy}
\p_t E \le C E + \int_{B(r)} f^2\, dx,
    \end{align}
where 
    \begin{align*}
C = 1 + \norm{\beta}_{L^\infty(0,T; W^{1,\infty}(\R^n; \R^n))}.
    \end{align*}
Indeed, using 
    \begin{align*}
\p_t u^2 = 2 u \p_t u 
= -2u\beta\cdot \nabla u + 2uf 
= 2uf -\beta\cdot \nabla u^2,
    \end{align*}
we get
    \begin{align*}
\p_t E 
&= 
\int_{B(r)} 2 u \p_t u\, dx + r' \int_{\p B(r)} u^2\, dx
= 
\int_{B(r)} 2uf -\beta\cdot \nabla u^2\, dx + r' \int_{\p B(r)} u^2\, dx
\\&= 
\int_{B(r)} 2 u f + u^2 \nabla \cdot \beta\, dx 
- \int_{\p B(r)} u^2 n_B \cdot \beta\, dx
+ r' \int_{\p B(r)} u^2\, dx
\\&\le
\int_{B(r)} 2 u f + u^2 \nabla \cdot \beta\, dx,
    \end{align*}
where we used $|n_B \cdot \beta| \le R = -r'$.

Now (\ref{energy}) follows from the Cauchy-Schwarz inequality.
We write
    \begin{align}\label{def_cone}
\mathcal C = \{t \in (0,T) : x \in B(r(t))\}.
    \end{align}
Gr\"onwall's inequality implies the following lemma.

\begin{lemma}\label{lem_L2}
Let $\beta \in L^\infty(0,T; W^{1,\infty}(\R^n; \R^n))$. Then there is $C > 0$
such that 
    \begin{align*}
\norm{u(t)}_{L^2(B(r(t)))} \le C \norm{f}_{L^2(\mathcal C)}, \quad t \in (0,T),
    \end{align*}
for all $u \in C^1(\overline{\mathcal C})$ satisfying (\ref{transport})
and $u|_{t=0} = 0$.
The constant $C$ depends only on $T$ and the norm of $\beta$.
\end{lemma}

We need also higher order energy estimates.
For notational simplicity, let us now consider only the case $n=1$. Let $k \in \N$. If $u$ satisfies (\ref{transport}), then $v = \p_x^k u$ satisfies
    \begin{align*}
\p_t v + \beta \p_x v = \tilde f, \quad \tilde f = \p_x^k f - [\p_x^k, \beta\p_x]u.
    \end{align*}
We write 
    \begin{align*}
\tilde E_k = \int_{B(r)} v^2\, dx,
\quad
E_K = \sum_{k=0}^K \int_{B(r)} (\p_x^k u)^2\, dx,
    \end{align*}
Applying (\ref{energy}) to $v$ gives
    \begin{align}\label{est_tildeE}
\p_t \tilde E_k \le C \tilde E_k + \int_{B(r)} \tilde f^2\, dx
\le C \tilde E_k + \int_{B(r)} (\p_x^k f)^2\, dx
+ C \tilde E_k,
    \end{align}
where the constant depends only on the $L^\infty(0,T; W^{k,\infty}(\R))$ norm of $\beta$.
The sum of the estimates (\ref{est_tildeE}) for $k=0,\dots,K$ gives  
    \begin{align}\label{energyK}
\p_t E_K \le C E_K + C\sum_{k=0}^K \int_{B(r)} (\p_x^k f)^2\, dx.
    \end{align}
Gr\"onwall's inequality gives the following lemma.

\begin{lemma}\label{lem_Hk}
Let $k \ge 1$ and $\beta \in L^\infty(0,T; W^{k,\infty}(\R))$. Then there is $C > 0$
such that 
    \begin{align*}
\norm{u(t)}_{H^k(B(r(t)))} 
\le 
C \sum_{j=0}^k \norm{\p_x^k f}_{L^2(\mathcal C)}, 
\quad t \in (0,T),
    \end{align*}
for all $u \in C^{k+1}(\overline{\mathcal C})$ satisfying (\ref{transport})
and $u|_{t=0} = 0$.
The constant $C$ depends only on $T$ and the norm of $b$.
\end{lemma}

We can apply a similar argument also to the time derivative. We will need only the first derivative in time $v = \p_t u$.
Then
    \begin{align*}
\p_t v + \beta \p_x v = \tilde f, \quad \tilde f = \p_t f - (\p_t\beta)\p_x u.
    \end{align*}
Writing $\tilde E = \int_{B(r)} v^2\, dx$ and applying (\ref{energy}) to $v$ gives
    \begin{align*}
\p_t \tilde E \le C \tilde E + C\int_{B(r)} \tilde f^2\, dx
\le C \tilde E + \int_{B(r)} (\p_t f)^2\, dx
+ C E_1,
    \end{align*}
where the constant depends only on $W^{1,\infty}((0,T) \times \R)$ norm of $\beta$.
The sum of this and (\ref{energyK}) with $K=1$, together with 
Gr\"onwall's inequality, gives the following lemma.

\begin{lemma}\label{lem_ptL2}
Let $\beta \in W^{1,\infty}((0,T) \times \R)$. Then there is $C > 0$
such that 
    \begin{align*}
\norm{\p_t u(t)}_{L^2(B(r(t)))} + \norm{u(t)}_{H^1(B(r(t)))} 
\le 
C \norm{f}_{H^1(\mathcal C)}, 
\quad t \in (0,T),
    \end{align*}
for all $u \in C^{2}(\overline{\mathcal C})$ satisfying (\ref{transport})
and $u|_{t=0} = 0$.
The constant $C$ depends only on $T$ and the norm of $\beta$.
\end{lemma}

\begin{corollary}\label{cor_stab_transp}
Let $\beta_1, \beta_2 \in W^{1,\infty}((0,T) \times \R)$
and suppose that 
    \begin{align*}
\norm{\beta_j}_{L^\infty((0,T) \times \R)} \le R,
\quad
\norm{\beta_j}_{W^{1,\infty}((0,T) \times \R)} \le R_1.
    \end{align*}
Let $u_j \in C^2(\overline{\mathcal C})$ satisfy $u_j|_{t=0} = 0$ and
$\p_t u_j + \beta_j \p_x u_j = f_j$, $j=1,2$.
Suppose furthermore that 
$\norm{f_j}_{H^1(\mathcal C)} \le R_1$.
Then 
    \begin{align*}
\norm{u_1 - u_2}_{H^1(\mathcal C)}
\le C\left(
\norm{f_1 - f_2}_{H^1(\mathcal C)} + \norm{\beta_1 - \beta_2}_{W^{1,\infty}(\mathcal C)} 
\right),
    \end{align*}
where the constant depends only on $T$ and $R_1$.
\end{corollary}
\begin{proof}
The function $w = u_1 - u_2$ satisfies
    \begin{align*}
\p_t w + \beta_1 \p_x w = f_1 - f_2 + (\beta_2 - \beta_1) \p_x u_2,
    \end{align*}
and from Lemma \ref{lem_ptL2},
    \begin{align*}
\norm{w}_{H^1(\mathcal C)} 
&\le 
C \norm{f_1 - f_2}_{H^1(\mathcal C)} + C\norm{\beta_1 - \beta_2}_{W^{1,\infty}(\mathcal C)} \norm{u_2}_{H^1(\mathcal C)}.
    \end{align*}
Using $\norm{f_2}_{H^1(\mathcal C)} \le R_1$, the claim follows from Lemma \ref{lem_ptL2}.
\end{proof}

\subsection{Local recovery}
We will now apply Corollary \ref{cor_stab_transp} to (\ref{eq_transp2}).

\begin{proposition}\label{prop_loc}
Let $\psi \in H^4(Q)$
and suppose that 
$\p_y \psi(p_0) \ne 0$ for a point
$p_0$ in $\Sigma_- \times (0,T)$.
Let $U_0 \subset \overline \Omega \times (0,T)$
be a neighbourhood of $p_0$.
Then there are a neighbourhood $\mathcal B \subset H^4(Q)$
of $\psi$, a neighbourhood $U \subset U_0$ of $p_0$
and a constant $C > 0$ such that 
for all $\tilde \psi \in \mathcal B$ there holds
    \begin{align*}
\norm{\Theta - \tilde \Theta}_{H^1(U)}
\le C \norm{\psi - \tilde \psi}_{H^4(U)},
    \end{align*}
where $\Theta$ and $\tilde \Theta$ are the solutions of (\ref{eq_transp})--(\ref{eq_Theta_bd}) with $(X, \zeta)=(X_\psi, \zeta_\psi)$ and $(X, \zeta)=(X_{\tilde \psi}, \zeta_{\tilde \psi})$, respectively.
\end{proposition}
\begin{proof}
The Sobolev embedding theorem implies that $\psi \in C^2(\overline Q)$.
In particular, the point value $\p_y \psi(p_0)$ is well-defined and there are a neighbourhood $V$ of $p_0$ and $\epsilon > 0$
such that 
$|\p_y \psi(p)| > \epsilon$ for all $p \in V$.
Define $\beta = -(\p_y \psi)^{-1} (\p_x \psi)$
in $V$, cf. (\ref{eq_transp2}), and 
    \begin{align*}
R = 1 + \norm{\beta}_{L^{\infty}(V)}.
    \end{align*}
Write $p_0 = (-a,y_0,t_0)$ and consider the set
    \begin{align*}
U 
= \{(x,y,t) \in Q : 
x \in (-a,-a+\delta),\ 
|y-y_0| < R(\delta - (x+a)),\ 
|t-t_0| < \delta\}
    \end{align*}
where $\delta > 0$ is small enough so that $\overline{U} \subset V \cap U_0$.
Choose a small enough neighbourhood $\mathcal B \subset H^4(Q)$
of $\psi$ so that $\tilde \beta = -(\p_y \tilde \psi)^{-1} (\p_x \tilde \psi)$ satisfies $\norm{\tilde \beta}_{L^{\infty}(V)} \le R$
for all $\tilde \psi \in \mathcal B$.
Now Corollary \ref{cor_stab_transp} implies the claimed estimate.
Indeed, 
    \begin{align*}
\norm{\beta - \tilde \beta}_{W^{1,\infty}(U)}
&\le C\norm{\psi - \tilde \psi}_{W^{2,\infty}(U)}
\le C \norm{\psi - \tilde \psi}_{H^4(U)}
    \end{align*}
and writing $f = (\p_y \psi)^{-1} \zeta$, with $\zeta= -\p_t \psi + \lambda \Delta \psi$, and defining $\tilde f$ analogously,
    \begin{align*}
\norm{f - \tilde f}_{H^1(U)}
&\le C \norm{\psi - \tilde \psi}_{W^{2,\infty}(U)}
+ C \norm{\psi - \tilde \psi}_{H^3(U)}.
    \end{align*}
Here the constants depend on 
    \begin{align*}
\sup_{\tilde \psi \in \mathcal B} \norm{\tilde \psi}_{H^4(Q)}
\quad \text{and} \quad 
\left(\inf_{\tilde \psi \in \mathcal B, p \in V} |\p_y \tilde\psi(p)|\right)^{-1}.
    \end{align*}
\end{proof}


\begin{remark}\label{rem_local}
Using notation from the above proposition,
it is immediate from equation \eqref{eq:vel_theta} that
\[
\norm{u - \tilde u}_{L^2(U)} \le \norm{\Theta - \tilde \Theta}_{H^1(U)}
\le C \norm{\psi - \tilde \psi}_{H^4(U)}, \quad \tilde \phi \in \mathcal B,
\]
where $(\phi, u)$ and $(\tilde \phi, \tilde u)$
both satisfy (\ref{eq_for_psi}) together with the boundary conditions (\ref{eq_phi_bd}) and $u_1=0$ on $\Sigma$.
Hence the map $\psi \mapsto u$ is locally Lipschitz continuous. \end{remark}

\subsection{Global recovery}

We recall the three cylinders inequality for the linearized
Navier--Stokes equation from \cite{LW20}, see also \cite{BIY16} for an earlier
related result. Let $u \in H^1(0,T;
H^2_{loc}(\Omega))$ be a nontrivial solution of \eqref{eq_NS} with
associated pressure $p \in L^2(0,T;H^1(\Omega))$. Let $B(x,R) \subset \R^2$ denote the ball of radius $R$ with the centre $x$.
\begin{theorem}[\cite{LW20}]\label{th_3cyl}
Let $C_0, T > 0$, $0 < R_1 < R_2 < R_3/3 < 1$, $x_0 \in \Omega$ and let $\epsilon > 0$ be small.
Suppose that $B(x_0,R_3) \subset \Omega$. Then there are $C > 0$ and $\alpha \in (0,1)$ such that
$$
\int_{\epsilon}^{T-\epsilon} \int_{B(x_0,R_2)} |u|^2 dx dt \le C 
\left( \int_{0}^{T} \int_{B(x_0,R_1)} |u|^2 dx dt \right)^{\alpha}
\left( \int_{0}^{T} \int_{B(x_0,R_3)} |u|^2 dx dt \right)^{1-\alpha}
$$
for all solutions to 
    \begin{align}\label{eq_NS_lin}
\p_t u - \nu \Delta u + (A \cdot \nabla) u + (u \cdot \nabla) B + \nabla p = 0,
\\\notag
\nabla \cdot u = 0,
    \end{align}
in $Q$ and all $A, B \in L^\infty(Q)$ satisfying 
$\norm{A}_{L^\infty(Q)} \le C_0$
and $\norm{B}_{L^\infty(Q)} \le C_0$.
\end{theorem}

This together with the local recovery implies the following global recovery.

\begin{theorem}\label{cor_NS_unique_cont}
Let $U \subset \R^2$ be open and suppose that $\overline U \subset \Omega$. 
Let $C_0 > 0$, $t_0 \in (0,T)$ and $y_0 \in (-b,b)$.
Define 
    \begin{align*}
\mathcal U = \{u \in C^\infty(Q) : 
\text{$u$ satisfies (\ref{eq_NS}), $u_1 = 0$ on $\Sigma$, and $\norm{u}_{L^\infty(Q)} \le C_0$}\}.
    \end{align*}
Let $u \in \mathcal U$ and let $\psi \in H^4(Q)$ satisfy (\ref{eq_phi})--(\ref{eq_phi_bd})
and $\p_y \psi(p_0) \ne 0$ where $p_0 = (-a, y_0, t_0)$.
Let $U_0 \subset \overline \Omega \times (0,T)$
be a neighbourhood of $p_0$.
Then there are a neighbourhood $\mathcal B \subset H^4(Q)$
of $\psi$
and constants $\alpha \in (0,1)$ and $C,s > 0$ such that 
if $\tilde u \in \mathcal U$ and if $\tilde \psi \in \mathcal B$ satisfies (\ref{eq_phi})--(\ref{eq_phi_bd}), with $u$ replaced by $\tilde u$,
then
    \begin{align*}
\norm{u - \tilde u}_{L^2((t_0-s,t_0 + s) \times U)}
\le C \norm{\psi - \tilde \psi}_{H^4(U_0)}^\alpha.
    \end{align*}
\end{theorem}
\begin{proof}
Proposition \ref{prop_loc} implies that there 
are a neighbourhood $\mathcal B \subset H^4(Q)$
of $\psi$ and a neighbourhood $U_1 \subset U_0$ of $p_0$ such that 
    \begin{align*}
\norm{u - \tilde u}_{L^2(U_1)}
\le C \norm{\psi - \tilde \psi}_{H^4(U_1)}, \quad \tilde \psi \in \mathcal B.
    \end{align*}
The difference $e = u - \tilde u$ 
satisfies (\ref{eq_NS_lin}) with 
$A=u$ and $B = \tilde u$. Indeed,
$$
(A \cdot \nabla) e + (e \cdot \nabla) B = 
(u \cdot \nabla) (u - \tilde u) + ((u - \tilde u) \cdot \nabla) \tilde u
= (u \cdot \nabla) u - (\tilde u \cdot \nabla) \tilde u.
$$
We can then apply Theorem \ref{th_3cyl} to $e$. Taking $B(x_0,R_1) \subset U_1$
it follows that for $R_2 > 0$ as in Theorem \ref{th_3cyl} there holds
$$
\int_{\epsilon}^{T-\epsilon} \int_{B(x,R_2)} |e|^2 dx dt \le C 
\left( \norm{\psi - \tilde \psi}^2_{H^4(U_0)}\right)^{\alpha}
\left( \|u\|_{L^2(Q)}^2+\|\tilde u\|_{L^2(Q)}^2  \right)^{1-\alpha}.
$$
The claim follows by iterating Theorem \ref{th_3cyl} finitely many
times (see for instance \cite{Rob91}). 
\end{proof}

\section{Concluding remarks}\label{sec:conclusion}
We have shown that for the SIV problem, the velocity
field $u$ is uniquely determined by the measured scalar field $\psi$, and $u$
depends continuously on $\psi$. The
stability is of H\"older type for the interior estimates
considered here. Due
to the nonlinearity of the map $\psi \mapsto u$, 
the scalar field in the
right hand side of the stability estimate of Theorem
\ref{cor_NS_unique_cont} is
measured in the $H^4$-norm. This is much stronger than the $L^2$-norm
of the velocities in the left hand side, but it seems unlikely that it
can be improved by much in the framework exposed here. The consequence
of this lack of balance in the estimate is that
computationally we must expect the error in the velocity to be larger
than that in the scalar field, even if $\alpha \approx 1$. 

An
outstanding challenge is to further develop the analysis so that it
allows for error estimates for a computational method. Several
building blocks for such a development are
available, for convection--diffusion equations and transport
in \cite{BNO20a, BNO20b}, for parabolic problems in \cite{BO18} and for the stationary
linearized Navier-Stokes' equation in \cite{BBFV20}. In those references
finite element methods are considered, but the arguments can be
reinterpreted in the context of spectral or Fourier methods that we
considered here.
\section*{Acknowledgment} 
EB was supported by EPSRC grants EP/P01576X/1 and EP/P012434/1, JJJG
was supported by the EPSRC grant EP/N024915/1 and LO was supported by EPSRC grants EP/L026473/1 and EP/P01593X/1.
\bibliographystyle{amsplain}

\bibliography{article}

\end{document}